\documentclass[a4paper,11pt]{amsart}
\usepackage[english]{babel}
\usepackage{amsmath,amssymb,amsthm,mathrsfs,amsopn,amscd}

\usepackage{verbatim,graphicx,color}
\usepackage[utf8]{inputenc}
\usepackage{indentfirst}
\usepackage[T1]{fontenc}
\usepackage{fancyhdr}
\usepackage{color}
\usepackage[usenames,dvipsnames,svgnames,table]{xcolor}

\newtheorem{thm}{Theorem}[section]
\newtheorem{lem}[thm]{Lemma}
\newtheorem{pro}[thm]{Proposition}

\newtheorem{oss}[thm]{Remark}

\newtheorem*{mt}{Main Theorem}

\newenvironment{example}[1][Example.]{\begin{trivlist}
\item[\hskip \labelsep {\bfseries #1}]}{\end{trivlist}}
\newenvironment{remark}[1][Remark.]{\begin{trivlist}
\item[\hskip \labelsep {\bfseries #1}]}{\end{trivlist}}

\def\R{\mathbb R}

\def\pa{\partial}

\newcommand{\sign}{\mathrm{sign\ }}

\title[Weighted perimeters and applications]{On a class of weighted Gauss-type isoperimetric inequalities and
applications to symmetrization}

\author[Marini]{Michele Marini}
\address{Scuola Normale Superiore, Piazza dei Cavalieri 7, 56126 Pisa, Italy}
\email{michele.marini@sns.it}

\author[Ruffini]{Berardo Ruffini}
\address{Institut Fourier, 100 rue des maths, BP 74, 38402 St Martin d'Hères cedex, France}
\email{berardo.ruffini@ujf-grenoble.fr}

\begin{document}

\begin{abstract}
We solve a class of weighted isoperimetric problems of the form 
\[
 \min\left\{\int_{\partial E}w e^V\,dx:\int_E e^V\,dx={\rm constant}\right\}
\]
where $w$ and $V$ are suitable functions on $\R^d$. As a consequence, we prove a comparison result for
the solutions of degenerate elliptic equations.
\end{abstract}

\maketitle

\section{Introduction}

In the celebrated paper \cite{T}, G. Talenti established several comparison results between the solutions of the
Poisson equation with Dirichlet boundary condition (with suitable data $f$ and $E$): 
\begin{equation}\label{problematalenti}
         -\Delta u=f  \mbox{ in $E$},\qquad        u=0  \mbox{ on $\pa E$}
\end{equation}
and the solutions of the corresponding problem where $f$ and $E$ are replaced by their spherical
rearrangements (see \cite[Chapter 3]{LL} for the definition and main properties of spherical rearrangement).
Precisely, he proves that if we denote by $v$ the solution
of the
problem with symmetrized data, then the rearrangement $u^*$ of the (unique) solution $u$ of \eqref{problematalenti}
is
pointwise bounded by $v$. Moreover he shows that the $L^q$ norm of $\nabla u$ is bounded, as well, by the $L^q$
norm of $\nabla v$, for $q\in(0,2]$. The proof of these facts  basically relies on two ingredients: the
Hardy-Littlewood-Sobolev
inequality and the isoperimetric inequality (see \cite{amfupa} and \cite{LL} for comprehensive
accounts on the
subjects). 

Later on, following such a scheme, many other works have been developed to prove analogous comparison results
related to the solutions of PDEs involving different kind of operators, see for instance
\cite{bbmp1,bbmp2,bbmp3,bcm1,bcm2,bla,blafeopos,cinesi} and the references therein. A recurring idea in these works is, roughly speaking, the following:
the operator considered is usually linked to a sort of {\em weighted perimeter}. Thus initially it is necessary 
to
solve a corresponding isoperimetric problem; then
the desired comparison results can be obtained following the ideas contained
in \cite{T}.

    \noindent
 For example in
\cite{bbmp2} the authors consider a class of weighted perimeters of the form
\[
 P_w(E)=\int_{\partial E}w(|x|)\,d\mathcal{H}^{d-1}(x),
\]
where $E$ is a set with Lipschitz boundary and $w:\R\to[0,\infty)$ a non-negative function, and prove, under
suitable convexity
assumptions on the weight $w$, that the ball centered at the origin is the unique solution of the mixed
isoperimetric problem
\[
 \min\{P_w(E):|E|={\rm constant}\}
\]
where $|\cdot|$ denotes the $d$-dimensional Lebesgue measure. As a consequence they prove comparison
results, analogous to those considered by Talenti in \cite{T}, for the solutions
of 
\[
-{\rm div}(w^2\nabla u)=f  \mbox{ in $E$},\qquad        u=0  \mbox{ on $\pa E$}.
\]
\noindent
Recently in \cite{bdr}, L. Brasco, G. De Philippis and the second author proved a quantitative
version of the weighted isoperimetric inequality considered in \cite{bbmp1}. Their proof is achieved
by means of a sort of {\em calibration technique}. One advantage of this technique is that it is adaptable
to other kind of problems, as that of considering other kind of functions in the weighted perimeter (e.g. 
Wulff-type weights, see \cite{bf}), or that of considering different measured spaces, as $\R^d$ endowed with
the Gauss measure.
\medskip

\noindent
In this paper we consider degenerate elliptic equations with Dirichlet boundary condition of
the form
\begin{equation}\label{problem}
         -\mathrm{div}(w^2\,e^{V}\nabla u)=f\,e^{V} \, \mbox{ in $E$},\qquad u=0 \mbox{ on $\pa E$}
\end{equation}
where $w$ and $V$ are two given functions, and we aim to prove 
analogous comparison results as those in \cite{T}. The particular form in which is written the measure $e^V$ is
due to the later applications, whose main examples are Gauss-type measures, that is $V(x)=-c|x|^2$. Bearing in
mind this instance, we consider a class of mixed isoperimetric problems of the form
\[
 \min\left\{P_{we^V}(E):\int_E e^{V}={\rm constant}\right\} 
\]
and prove, by means of a calibration technique reminiscent of that developed in \cite{bdr}, that the solutions,
under suitable
assumptions on $V$ and $w$, are half-spaces, see Proposition \ref{semispazides} and Theorem \ref{brasco0}.
Then, using a suitable concept of rearrangement related to the measures considered, 
we
prove, in the Main Theorem in Section \ref{main}, comparison results between the solutions of
\eqref{problem} and the solutions of the same
equation with rearranged data.

\section{Preliminaries on rearrangement inequalities}\label{prerequisiti}

In this section we introduce the main definitions and properties about the concept of symmetrization and
rearrangement we shall make use of. 

\noindent
Let $\mu$ be a finite Radon measure on $\R^d$, a \emph{right rearrangement} with respect to $\mu$ 
is defined, for any Borel set $A$, as
%
\[
R_A^\mu=\{(x_1,x')\in\R\times\R^{d-1}\,:\,x_1> t_A\},
\]
where $t_A=\inf\left\{t\,: \mu(A)=\mu(\{(x_1,x')\in\R\times\R^{d-1}\,:\,x_1> t\})\,\right\}$. Notice that if 
$d\mu=fdx$, for some positive and measurable function $f$, then the value of $t$ is uniquely determined.\\
\noindent
Given a non-negative Borel function $f:\R^d\to[0,+\infty)$, we call \emph{right increasing rearrangement} of $f$
the
function $f^{*\mu}$ given by
\[
f^{*\mu}(x)=\int_0^{+\infty}\chi_{R^\mu_{\{f>t\}}}(x)\,dt
\]
where $\chi_A$ is the characteristic function of the set $A$. As an aside we notice that the right increasing
rearrangement of the characteristic function of a Borel set 
$A$ coincides with the characteristic function of $R_A^\mu$. Clearly $f^{*\mu}$ is non-negative, increasing with
respect to
the first variable $x_1$, and constant on the sets $\{(x_1,x')\in\R\times\R^{d-1}:x_1=t\}$, for
$t\in\R$.
Moreover $f$ and $f^{*\mu}$ share the same distribution function:
\[
\mu_f(t):=\mu(\{f>t\})=\mu(\{f^{*\mu}>t\})=\mu_{f^{*\mu}}(t).
\]
We furthermore define $f^{\star\mu}:\R^+\rightarrow\R^+$ as the smallest decreasing function satisfying
$f^{\star\mu}(\mu_f(t))\geq t$; in other words
\[
f^{\star\mu}(s)=\inf\{t>0\,:\, \mu_f(t)<s\}.
\]
It is useful to bear in mind that $\{s:f^{\star\mu}(s)>t\}=[0,\mu_f(t)]$ so that by the Layer-Cake Representation
Theorem (see for instance \cite{LL}) we have 
\begin{equation}\label{symlc}
\int_0^{\mu(\{x_1>t\})}f^{\star\mu}(s)\,ds=\int_{t}^\infty \mu_f(s)\,ds=\int_{\{x_1>t\}} f^{*\mu}(x)\,dx.
\end{equation}

\noindent
We conclude this section by proving the {\it Hardy-Littlewood}
rearrangement inequality related to the right symmetrization.
\begin{lem}[Hardy-Littlewood rearrangement inequality]\label{hardy}
Let $f$ and $g$ be non-negative Borel functions from $\R^d$ to $\R$. Then for any non-negative Borel measure
$\mu$ we have
\[
\int_{\R^d}f\,g\,d\mu\leq\int_{\R^d}f^{*\mu}g^{*\mu}d\mu.
\]
\end{lem}
\begin{proof}
We have
\[
\begin{aligned}
\int_{\R^d}f\,g\,d\mu&=\int_{\R^d}\int_0^\infty\int_0^\infty \chi_{\{f>t\}}(x)\chi_{\{g>s\}}(x)\,dt\,ds\,d\mu(x)\\
&=\int_0^\infty\int_0^\infty\int_{\R^d}
\chi_{\{f>t\}\cap \{g>s\}}(x)\,d\mu(x)\,dt\,ds\\
&=\int_0^\infty\int_0^\infty\mu(\{f>t\}\cap\{g>s\})\,dt\,ds
\\
&\le\int_0^\infty\int_0^\infty\min(\mu(\{f>t\}),\,\mu(\{g>s\}))\,dt\,ds
\\&
=\int_0^\infty\int_0^\infty\min(\mu(\{f^{*\mu}>t\}),\,\mu(\{g^{*\mu}>s\}))\,dt\,ds
\\
&=\int_0^\infty\int_0^\infty\mu(\{f^{*\mu}>t\}\cap\{g^{*\mu}>s\})\,dt\,ds=
\int_{\R^d}f^{*\mu}\,g^{*\mu}\,d\mu,
\end{aligned}
\]
where we used the fact that $\{f^{*\mu}>t\}$ and $\{g^{*\mu}>s\}$ are half-spaces of the form
$\{(x_1,x')\in\R\times\R^{d-1}:x_1>r\}$ for some $r\in\R$ and so
\[
 \min(\mu(\{f^{*\mu}>t\}),\,\mu(\{g^{*\mu}>s\}))=\mu(\{f^{*\mu}>t\}\cap\{g^{*\mu}>s\}).
\]

\end{proof}
\begin{remark}
Setting $g=\chi_A$ in Lemma \ref{hardy} and thanks to \eqref{symlc} we get
\begin{equation}\label{hardy2}
\int_A f\,dx\leq \int_{R_A^\mu}f^{*\mu}(x)\,dx=\int_0^{\mu(A)}f^{\star\mu}(s)\,ds.
\end{equation}

\end{remark}


%
\section{A class of weighted isoperimetric inequalities}

Given a measurable function $V:\R^d\to \mathbb R$ we denote by $\mu[V]$ the absolutely continuous measure whose
density equals $e^ V$, that is, for any measurable set $E\subset\R^d$ 
\[
  \mu[V](E)=\int_E e^{V(x)} dx;
\]
in what follows with the scope of simplifying the notation, and if there is no risk of confusion, we will
drop the dependence of $V$, writing $\mu$ instead of $\mu[V]$.
Moreover we will often adopt the notation $x=(x_1,x')\in\R\times\R^{d-1}$ and denote by $R_A$ instead of
$R_A^{\mu[V]}$ the right rearrangement of $A$ with respect to the measure $\mu[V]$. Given a Borel {\it
weight} function $w:\mathbb R\to\mathbb [0,+\infty]$ we define, for any open set $A$ with Lipschitz boundary, the
following concept of {\em weighted perimeter}:
\[
P_{w,V}(A)=\int_{\pa A}w(x_1)e^{V(x)}d\mathcal H^{d-1}(x).
\]
In the following proposition we show that, under suitable conditions on $w$ and $V$, the half-spaces of the form
$\{(x_1,x'):x_1>t\}$ are
the only minimizers of the weighted perimeter among the sets of fixed volume with respect to the measure
$\mu[V]$.

\begin{pro}\label{semispazides}
Let $A\subset\R^d$ be a set with Lipschitz boundary. Suppose that $w : \mathbb R\to\mathbb R^+$ and
$V:\R^d\to\mathbb R$ are
$C^1$-regular functions
satisfying the following assumptions:
\begin{itemize}
\item[{\it(i)}] $\mu(A)=\mu(R_A)<+\infty$,
\item[{\it (ii)}] the function $\pa_1 V(x)$ depends only on $x_1$ and  $g(x):=-w'(x_1)-w(x_1)\pa_1 V(x)$ is a
non-negative decreasing function on the
real line.
\end{itemize}
Then 
\begin{equation}\label{disprincipale}
P_{w,V}(A)\geq P_{w,V}(R_A).
\end{equation}
\end{pro}

\begin{proof}
We start by noticing that if $P_{w,V}(A)=+\infty$ there is nothing to prove. Hence we can suppose that
\begin{equation}\label{perimetrofinito}
 P_{w,V}(A)<+\infty.
\end{equation}
Let $e_1=(1,0,\dots,0)\in\R^d$ and consider the vector field $-e_1w(x_1)e^{V(x)}$. Its divergence is given by
\[
\mathrm{div} (-e_1w(x_1)e^V(x))=(-w'(x_1)-w(x_1)\pa_1 V(x))e^{V(x)}=g(x)e^{V(x)}.
\]
By an application of the Divergence Theorem we have 
\begin{equation}\label{perimestimate}
\begin{aligned}
\int_Ag(x)d\mu(x)&=\int_{A}\mathrm{div} (-e_1w(x_1)e^{V(x)})dx\\
&=\int_{\pa A}w(x_1)e^{V(x)}\langle\nu_A(x),-e_1\rangle d\mathcal H^{d-1}(x)\\
&\le\int_{\pa A}w(x_1)e^{V(x)}d\mathcal H^{d-1}(x)=P_{w,V}(A),
\end{aligned}
\end{equation}
\noindent
where $\nu_A(x)$ is the outer unit normal to $\partial A$ at $x$.
Let $t_A$ be a real number such that the right half-space $R_A=\{(x_1,x'):x_1\ge t_A\}$ satisfies
$\mu(R_A)=\mu(A)$. Then, since the outer normal of $R_A$ is the constant vector field $-e_1$, 
the
inequality in \eqref{perimestimate} turns into an equality 
if we replace  $A$ with $R_A$. Notice that by condition $(ii)$ and \eqref{perimestimate} we have
\[
 P_{w,V}(R_A)=\int_{R_A\setminus A}g\,d\mu+\int_{R_A\cap A}g\, d\mu\le g(t_A)\mu(A)+P_{w,V}(A).
\]
Thanks to assumption $(i)$ and \eqref{perimetrofinito} such quantities are finite and so we get
\[
P_{w,V}(A)-P_{w,V}(R_A)\geq\int_Ag(x)d\mu(x)-\int_{R_A}g(x)d\mu(x).
\]
Since, by definition, $\mu(A)=\mu(R_A)<+\infty$ again by condition $(i)$ we obtain $\mu(A\setminus
R_A)=\mu(R_A\setminus A)<+\infty$.
Thus
\begin{equation}\label{finito}
\begin{aligned}
\int_Ag(x)&d\mu(x)-\int_{R_A}g(x)d\mu(x)=\int_{A\setminus R_A}g(x)d\mu(x)-\int_{R_A\setminus
A}g(x)d\mu(x)\\
&=\int_{A\setminus R_A}(g(x)-g(t_Ae_1))d\mu(x)-\int_{R_A\setminus A}(g(x)-g(t_Ae_1))d\mu(x).
\end{aligned}
\end{equation}
Since every $x\in A\setminus R_A$ (respectively $x\in R_A\setminus A$) satisfies $\langle x,e_1\rangle<t_A$
(respectively $\langle x,e_1\rangle>t_A$), by condition {\it (ii)} we deduce
\begin{equation}\label{quantitativa}
\begin{aligned}
P_{w,V}(A)-P_{w,V}(R_A)&\ge\int_{A\setminus R_A}|g(x)-g(t_Ae_1)|d\mu(x)+\int_{R_A\setminus
A}|g(x)-g(t_Ae_1)|d\mu(x)
\\&=\int_{A\Delta R_A}|g(x)-g(t_Ae_1)|d\mu\geq 0,
\end{aligned}
\end{equation}
where $A\Delta R_A=(A\setminus R_A)\cup (R_A\setminus A)$ stands for the symmetric difference between $A$ and
$R_A$. This concludes the proof.
\end{proof}

\begin{oss}[Necessity of the assumptions]\rm
We stress that the integrability condition $(i)$ is necessary to  formulas \eqref{perimestimate} and \eqref{finito}
(and thus to our proof) to
work. 

\noindent
Concerning condition $(ii)$, we note that it is needed just for technical reasons. 
Nonetheless we stress that our proof offers a slightly stronger inequality than \eqref{disprincipale}. Indeed the
right-hand side of \eqref{quantitativa} may be seen as a modulus of continuity of the $L^1$ distance between $A$
and $R_A$. Thus it would be interesting to understand how much our hypotheses are far from optimality
(compare also with \cite[Remark $2.3$]{bdr}). 
\end{oss}

\begin{oss}[Equality cases]\rm
An inspection of the proof of Proposition \ref{semispazides}, and in particular of inequality
\eqref{perimestimate}, shows that if $w>0$, then we have equality in \eqref{disprincipale} only if $A$ is
equal to the half space $R_A$, up to set of
zero $d$-dimensional Lebesgue measure. On the other hand, if the set $\{w=0\}$ has positive Lebesgue measure, we
can not expect any kind of uniqueness for the equality cases of such an inequality. 
\end{oss}

\begin{example}\label{esempione}\rm
A non-trivial example fulfilling condition $(ii)$ of Proposition \ref{semispazides} is the
following
\[
V(x_1,x')=-c(x_1|x_1|+|x'|^2),\quad w(x_1)=e^{-ax_1},
\]
with $a,c>0$ constants satisfying $a^2-2c\ge0$. To prove this fact we initially observe that if $x_1\ne0$ such a
condition
is equivalent to require that
\begin{equation}\label{oder}
w''(x_1)+V_1''(x_1)w(x_1)+V_1'(x_1)w'(x_1)\ge0
\end{equation}
which turns out to be equivalent, in our example, to
\[
 a^2-2c+2ac|x_1|\ge0.
\]
Then, since $-w'(x_1)-w(x_1)\partial_1 V(x_1)$ is continuous in $x_1=0$,
condition
$(ii)$ is satisfied everywhere.

\end{example}
\bigskip
To transform inequality \eqref{disprincipale} into a well posed isoperimetric problem, it would be more advisable
 to eliminate the integrability hypothesis $(i)$ in Proposition \ref{semispazides} by requiring that the measure
 $\mu(\R^d)<+\infty$. This fact, together with ordinary differential inequality required in assumption $(ii)$, is
 seldom satisfied.  
 Hence, to get other instances of functions which fulfill
 inequality
  \eqref{oder} together with the integrability property {\it (i)} of Proposition
 \ref{semispazides} it is worth restricting our attention to the half-space
 \[
  \R^d_+=\{(x_1,x')\in\R\times\R^{d-1}: x_1>0\}.
 \]
 As an immediate corollary of Proposition \ref{semispazides} we get that 
the solution of the problem
  \begin{equation}\label{errore}
   \min\left\{P_{w,V}(A): A\subseteq\R^d_+,\,\,\mu(A)=c, \,\,\pa A\,\,{\rm Lipschitz}\right\}
  \end{equation}
\noindent
is given by $R_c=\{x_1\ge t_c\}$ where $t_c$ is such that $\mu(R_c)=c$.\\

\begin{oss}\label{osservazionciona2}\rm
Notice that the non-mixed Gauss case, $w$ constant and $V(x)={-c|x|^2}$, is not covered by our hypotheses.
Nevertheless in this case examples of functions $w$ which satisfy the hypotheses of Proposition
\ref{semispazides} are given by $w(t)=t^{-a}$ with $a\ge1$ or $w(t)=b+e^{-at}$,  with $a,b\ge 0$ such that
$a^2-2c(1+b)>0$ (as can be easily seen reasoning as in the previous example). In the latter case at least if
$b=0$ we have that
\[
 we^V=e^{a^2/(4c)}\exp{\left(-c\left|x+{\bf e_1}\frac{a}{2c}\right|^2\right)},
\]
where ${\bf e_1}=(1,0,\dots,0)\in\R^d$, 
which can be rephrased\footnote{as suggested us by an anonymous Referee.}
as
the fact that
the solutions of the isoperimetric problem in the half-space $\R^d_+$ with
(suitable) mixed Gaussian conditions
\[
 \min\left\{P_{\gamma_{\sigma,\eta}}(E): \gamma_{\sigma,0}(E)={\rm constant},\,\, E\subseteq \R^d_+,\,\,\partial E
\,\,{\rm Lipschitz} \right\}
\]
are right-half spaces.  Here we
denoted by $\gamma_{\sigma,\eta}$ the normal distribution whose covariance matrix is
$\sigma \rm{Id}$ and whose mean vector $\eta$ is given 
by $\eta=-\frac{a}{2c}{\bf e_1}$. If $b\ne0$ the unique change is that the perimeter is weighted by means of the
sum of two Gaussian measures.
We recall that, as pointed out in the Introduction, similar problems related to the Gauss measure are considered in
\cite{bbmp3,bcm1,bla,blafeopos,cinesi}.
\end{oss}

Notice that we defined the perimeter $P_{w,V}$ only for sets with Lipschitz boundary, 
but for our later applications it will be useful to have a definition of perimeter which comprehends also less
regular
subsets of $\R^d$.
A measurable set $A$ is said to have locally finite (Euclidean) perimeter (we refer to \cite{M} for a complete
overview on the subject) if there exists a vector-valued Radon measure $\nu_A$ called \emph{Gauss--Green measure}
of the set $A$ such that, for every $T\in C_c^1(\R^d;\R^d)$, it holds true that
\[
\int_A \mathrm{div} T =\int_{\R^d}\langle T,d\nu_A\rangle.
\]
The perimeter of $A$ is defined in terms of the total variation of the Gauss--Green measure of $A$ as
$P(A)=|\nu_A|(\R^d)$.
For any set $A$ of locally finite perimeter we then define the \emph{weighted perimeter} $P_{w,V}$ by
\[
P_{w,V}(A)=we^V|\nu_A|(\R^d).
\]
Since when $A$ has Lipschitz boundary $|\nu_A|=\mathcal H^{d-1}\llcorner \pa A$, the above definition is
coherent
with the one given at the beginning of this section on such sets. \\


\begin{thm}\label{brasco0}
Let $w$ and $V$ non-negative and $C^1$-regular functions satisfying condition $(ii)$ of Proposition
\ref{semispazides}.  Suppose moreover that $\mu(\R_+^d)<+\infty$; then the problem
\[
\min\left\{P_{w,V}(A): A\subseteq\R^d_+,\,\,\mu(A)=c\right\}
\]
admits a solution, and this solution coincides with the one of \eqref{errore}.
\end{thm}

\begin{proof}
Let $A$ be a measurable set of locally finite perimeter and suppose, by contraddiction, that
$P_{w,V}(A)<P_{w,V}(R_A)$. We start by noticing that $P_{w,V}(R_A)<+\infty$, indeed, recalling
\eqref{perimestimate} we have that 
\[
P_{w,V}(R_A)=\int_{R_A}g(x)\,d\mu(x)\le g(0)\mu(A).
\]
By \cite[Theorem II.2.8]{M} we can find  a sequence of  sets $A_n$ with smooth boundary such that
$\chi_{A_n}\rightarrow\chi_A$ in $L_{\rm{loc}}^1(\R^d)$ and $|\nu_{A_n}| \rightharpoonup ^* |\nu_A|$, where
$\rightharpoonup ^*$ indicates the weak* convergence of Radon measures.
Since $\mu(\R^d_+)<+\infty$, we also have that 
\begin{equation}\label{convvol}
\chi_{A_n}\rightarrow\chi_A \,\,\,\, \mathrm{in }\, L^1(\R^d, \mu)
\end{equation}
 and, since $we^V$ is a continuous function 
\begin{equation}\label{convper}
we^V|\nu_{A_n}| \rightharpoonup ^* we^V|\nu_A|.
\end{equation}
Thanks to \eqref{convper} we get
\[
P_{w,V}(A)=\lim_{n\to\infty} P_{w,V}(A_n)\ge \lim_{n\to\infty} P_{w,V}(R_{A_n}).
\]
We are left to show that $\lim_{n\to\infty} P_{w,V}(R_{A_n})=P_{w,V}(R_A)$, but
\[
|P_{w,V}(R_A)-P_{w,V}(R_{A_n})|\le g(0)|\mu(A)-\mu(A_n)|,
\]
and we can conclude thanks to \eqref{convvol} and the fact that $\mu(\R^d_+)<+\infty$.
\end{proof}

\section{Main result}\label{main}
%
%
In this section we consider sets $E\subseteq\R^d_+$ and we define $d\mu=e^V\,dx$,
$R_E=\{x_1>t_E\}$ where
$t_E\in\R$ is such that $\mu(R_E)=\mu(E)$ and $f^*=f^{*\mu}$ the
right rearrangement of a function $f$ with respect to $\mu$. In what follows we consider
problems of the form
\begin{equation}\label{brasco}
\left\{ \begin{array}{ll}
         -\mathrm{div}(w^2\,e^{V}\nabla u)=f\,e^{V} & \mbox{in $E$}\\
        u=0 & \mbox{on $\pa E$}\end{array} \right.
\end{equation}
which must be intended in weak sense. Precisely,  a solution of \eqref{brasco} is a function $u\in
H^1_0(e^V,w^2e^V,E)$, defined as the space of functions in $L^2(E,e^V)$ whose weak gradients are in 
$L^2(E,w^2e^V)$ which
vanish on the boundary of $E$ in the trace sense\footnote{which is possible since the regularity of
$w$ and $V$ and if $E$ has Lipschitz boundary.}, and which satisfies 
\begin{equation}\label{weakproblem0}
\int_E\langle \nabla u,\nabla \phi\rangle w^2e^V\,dx=\int_E f\,\phi\,e^V\,dx
\end{equation}
 for any $\phi\in H^1_0(e^V,w^2e^V,E)$. \\
The main scope of this section is to prove {\em a priori} estimates for the solutions of problem \eqref{brasco}.
For this reason we shall always consider that a solution $u$ exists. Clearly this requirement depends on the
choice of $w$, $V$ and $f$. General instances of such functions for which the existence of a solution for
problem \eqref{brasco} is guaranteed, can be found in \cite{tru} (see also \cite{bla,cinesi,bbmp3,blafeopos}). Here
we limit
ourselves to state that most of the examples considered in Remark \ref{osservazionciona2}, as the {\em
mixed-Gaussian case}
 $V(x)=-c|x|^2$, $w(t)=b+e^{-at}$ with $a^2- 2c(1 + b) > 0$ and $b$ strictly
positive, are covered by the cases considered in \cite{tru}, whenever $f\in
L^2(E,e^V)$. 

\begin{mt}\label{mainthm}
Suppose that the set $E\subset\R^d_+=\{(x_1,x'):x_1>0\}$ and the functions $w:[0,+\infty]\to(0,+\infty]$  and
$V:\R^d \to\R$ satisfy the hypotheses of Proposition \ref{semispazides}. Consider the two problems
\begin{equation}\label{problemthm}
\left\{ \begin{array}{ll}
         -\mathrm{div}(w^2\,e^{V}\nabla u)=f\,e^{V} & \mbox{in $E$}\\
        u=0 & \mbox{on $\pa E$}\end{array} \right.
\end{equation}
and 
\begin{equation}\label{problemsym}
\left\{ \begin{array}{ll}
         -\mathrm{div}(w^2\,e^{V}\nabla v)=f^{*}e^{V} & \mbox{in $R_E$}\\
        v=0 & \mbox{on $\pa R_E$}\end{array} \right.
\end{equation}
where $0<f\in L^2(\R^d_+,\mu)$.
Then the problem \eqref{problemsym} has as solution the one variable function $v(z)$  given by 
\begin{equation}\label{v}
 v((z,z'))=v(z)=\int_{\mu(\{x_1\ge z\})}^{\mu(R_E)}\frac{1}{h^2(s)}\left(\int_0^sf^*(\xi)\,d\xi\right)\,ds,
\end{equation}
where
\begin{equation}\label{acca}
 h(m)=w(\Phi^{-1}(m))\int_{\R^{d-1}}\mu(\Phi^{-1}(m),x')\,dx',
\end{equation}
being $\Phi(t)=\mu(\{x_1>t\})$.
Moreover, for any solution $u$  of the problem \eqref{problemthm}, we have
%
\begin{equation}\label{tesi1}
 u^*(x)\le v(x),
\end{equation}
and, for any $q\in(0,2]$,  
\begin{equation}\label{tesi2}
 \int_E |\nabla u|^qw^q \,d\mu \le \int_{R_E} |\nabla v|^qw^q \,d\mu
\end{equation}

\end{mt}


\begin{proof}
Let us suppose for the moment that the function $v$ given in \eqref{v} is a solution for the problem
\eqref{problemsym}. To prove
\eqref{tesi1} and \eqref{tesi2} we consider the functions $\phi_h$ defined as
\[
\phi_h(x)=\left\{ \begin{array}{ll}
         \sign(u) & \mbox{if $|u|>t+h$}\\
	 \frac{u(x)-t\sign u(x)}{h} & \mbox{if $|u|\in [t,t+h)$}\\
         0 & \mbox{if $|u|<t$},\end{array} \right.
\]
where $0 \leq t< \mathrm{ess\,sup}|u|$ and $h >0$. Notice that, for every $h>0$, $\phi_h$ is an admissible test function, since the solution $u$ belongs to the space $H^1_0(e^V,w^2e^V,E)$. Then \eqref{weakproblem0} turns into
\[
\frac 1 h \int_{\{|u|\in[t,t+h)\}}\langle \nabla u,\nabla u\rangle w^2\,d\mu=\frac 1 h\int_{\{|u|\in[t,t+h)\}} f\,
(u-t\frac{u}{|u|}) d\mu+\int_{\{|u|>t+h\}}f\,\sign(u)\,d\mu.
\]
Taking the limit for $h\to 0$, we get
\begin{equation}\label{stima1}
-\frac{d}{dt}\int_{\{|u|>t\}}|\nabla u|^2w^2\,d\mu=\int_{\{|u|>t\}}f\,d\mu.
\end{equation}
Let us analyze the left-hand side of equation \eqref{stima1}. We claim that the following
inequality holds true for almost every $t$:
\begin{equation}\label{s2}
-\frac{d}{dt}\int_{\{|u|>t\}}|\nabla u|^2w^2\,d\mu\ge\frac{\left(-\frac{d}{dt}
\int_{\{|u|>t\}}|\nabla u|w\,d\mu\right)^2}{-\mu'_u(t)},
\end{equation}
where $\mu_u(t)$ is the distribution function of $u$ introduced in the Section \ref{prerequisiti}.\\
\noindent
Indeed $\mu_u(t)$ is a decreasing function and thence it is derivable for almost every $t$, thanks to the H\"older inequality we get
\[
 \begin{aligned}
  -\frac{d}{dt}\int_{\{|u|>t\}}&|\nabla u|w\,d\mu=\lim_{h\rightarrow 0}\frac 1 h \int_{t<|u|<t+h}|\nabla
u|w\,d\mu\\
&\le\lim_{h\rightarrow 0}\left(\int_{\{t<|u|<t+h\}}|\nabla u|^2w^2\,d\mu\right)^{1/2}\left(\int_{\{t<|u|<t+h\}}\frac{1}{h^2}\,d\mu\right)^{1/2}\\
&=\lim_{h\rightarrow 0}\left(\frac 1 h\int_{\{t<|u|<t+h\}}|\nabla u|^2w^2\,d\mu\right)^{1/2}\left(\frac 1 h\int_{\{t<|u|<t+h\}} 1\,d\mu\right)^{1/2}\\
&=\left(-\frac{d}{dt}\int_{\{|u|>t\}}|\nabla u|^2w^2\,d\mu\right)^{1/2}\left(-\mu_u'(t)\right)^{1/2}
 \end{aligned}
\]
By the Co-Area formula and the fact that $w$ is strictly positive and $C^1$-regular, we easily get that the set
$\{u>t\}$ is a set of locally finite (Euclidean) perimeter. Thus, thanks to Proposition \ref{semispazides} and
Theorem \ref{brasco0} we get
\begin{equation}\label{stima2}
-\frac{d}{dt}\int_{\{|u|>t\}}|\nabla
u|w\,d\mu=\int_{\{|u|=t\}}w\,d\mu=P_{w,V}(\{|u|>t\})\ge P_{w,V}({\{u^*>t\}}).
\end{equation}
We introduce the function
\begin{equation}
 \Phi(t)=\mu(\{x_1>t\}).
\end{equation}
We recall that the weight function $w$ is constant on the boundary of the super level sets of $u^*$, so that the
perimeter of $\{u^*>t\}$ can be written as
\[
 P_{w,V}(\{u^*>t\})=w(\tau)\int_{\R^{d-1}}\mu(\tau,x')\,dx'.
\]
Moreover $\tau\in\R$ satisfies $\mu_{u^*}(t)=\Phi(\tau)$ that is $\tau=\Phi^{-1}(\mu_{u^*}(t))$ (notice
that $\Phi$ is a strictly decreasing function and thus invertible) so that we
can write the previous formula as
\begin{equation}\label{stima3}
 P_{w,V}(\{u^*>t\})=w(\Phi^{-1}(\mu_{u^*}(t)))\int_{\R^{d-1}}\mu(\Phi^{-1}(\mu_{u^*}(t)),x')\,
dx':=h(\mu_{u^*}(t)).
\end{equation}
Plugging \eqref{stima2} in \eqref{s2}, and recalling \eqref{stima3} we get that
\begin{equation}\label{stima4}
-\frac{d}{dt}\int_{\{|u|>t\}}|\nabla u|^2w^2\,d\mu\ge\frac{h(\mu_{u^*}(t))^2}{-\mu_{u^*}'(t)}. 
\end{equation}
We pass now to estimate the right-hand side of \eqref{stima1}: equation \eqref{hardy2} with $A=\{|u|>t\}$ turns
into
\begin{equation}\label{stima5}
\int_{\{|u|>t\}}f\,d\mu\le \int_{\{|u^*|>t\}}f^*\,d\mu=\int_0^{\mu_{u^*}(t)}f^\star(s)\,ds.
\end{equation}
Combining \eqref{stima5} and \eqref{stima4} we get
\begin{equation}\label{stima6}
 \frac{\left(\int_0^{\mu_{u^*}(t)}f^\star(s)\,ds\right)\mu_{u^*}'(t)}{h^2(\mu_{u^*}(t))}\le-1.
\end{equation}
Reasoning analogously for the function $v$, we easily see that, since $v$ is constant on every set $\{x_1=t\}$ and
since $v=v^*$, \eqref{stima6} holds for $v$ as an equality. Consider now the real function
\[
 F(r)=\frac{\int_0^r f(s)\,ds}{h(r)^2},
\]
and let $G$ be a primitive of $F$. Since $F\ge0$, we have that $G$ is increasing. Moreover by
our previous
analysis  we have that
\[
F(\mu_{u^*}(t))\mu_{u^*}'(t)\le-1= F(\mu_v(t))\mu_v'(t).
\]
We recall that here $\mu_{u^*}'(t)$ denotes the derivative almost everywhere of the function $\mu_{u^*}(t)$.
Moreover $t\mapsto G(\mu_{u^*}(t))$ is a monotone non-increasing function which satisfies the chain rule in any
point of differentiability of $\mu_{u^*}$, so that, by \cite[Corollary $3.29$]{amfupa},
we get that 
\begin{equation}\label{palle1}
     G(\mu_{u^*}(t))\le
G(\mu_{u^*}(0))+\int_0^tF(\mu_{u^*}(\tau))\mu_{u^*}'(\tau)\,d\tau.
\end{equation}
On the other hand, being $\mu_v(t)$
an absolutely continuous function (since $v$ is a $C^1$-regular with positive derivative one variable function) we 
have 
\begin{equation}\label{palle2}
G(\mu_{v}(t))=G(\mu_{u^*}(0))+ \int_0^tF(\mu_{v}(\tau))\mu_{v}'(\tau)\,d\tau,
\end{equation}
so that, since $G(\mu_{v}(0))=G(\mu_{u^*}(0))$, we get that $G(\mu_{u^*}(t))\le G(\mu_v(t))$. This implies that
$\mu_{u^*}(t)\le\mu_v(t)$ for any $t$ and
hence that $u^*\le v$, since $u^*$ and $v$ depends only on $x_1$ and are increasing functions of such a variable.
 

We pass now to the proof of \eqref{tesi2}. Using the H\"older inequality and
reasoning as before we obtain, for  $0< q\le2$,  
\[
\begin{aligned}
-\frac{d}{dt} \int_{\{|u|>t\}}&|\nabla u|^q w^q\,d\mu =\lim_{h\rightarrow 0}\frac 1 h \int_{\{t<|u|<t+h\}}|\nabla
u|^q w^q\,d\mu\\
& \le  \lim_{h\rightarrow 0}\left(\frac 1 h \int_{\{t<|u|<t+h\}}|\nabla u|^2w^2\,d\mu\right)^{q/2}\left(\frac 1 h\int_{\{t<|u|<t+h\}}d\mu\right)^{1-q/2}\\
&=\left(-\frac{d}{dt}\int_{\{|u|>t\}}|\nabla u|^2w^2\,d\mu
\right)^{q/2}(-\mu_u'(t))^{1-q/2}.
\end{aligned}
\]
Recalling \eqref{stima1} and \eqref{stima5} we have
\[
-\frac{d}{dt}\int_{\{|u|>t\}}|\nabla u|^2w^2\,d\mu
\leq \int_0^{\mu_{u^*}(t)}f^*(s)\,ds,
\]
 thus
\begin{equation}\label{stima9}
-\frac{d}{dt} \int_{\{|u|>t\}}|\nabla u|^q w^q\,d\mu\le
\left(\int_0^{\mu_{u^*}(t)}f^\star(s)\,ds\right)^{q/2}(-\mu_u'(t))^{1-q/2}.
\end{equation}
Combining \eqref{stima9} and \eqref{stima6} we finally get
\[
-\frac{d}{dt} \int_{\{|u|>t\}}|\nabla u|^q w^q\,d\mu\le
(-\mu_{u^*}'(t))\left(h(\mu_{u^*}(t))^{-1}\int_0^{\mu_{u^*}(t)}f^\star(s)\,ds\right)^q.
\]
By integrating on both side between $0$ and $+\infty$, we get
\[
\int_E |\nabla u|^q w^q\,d\mu\le\int_0^\infty
(-\mu_{u^*}'(t))\left(h(\mu_{u^*}(t))^{-1}\int_0^{\mu_{u^*}(t)}f^\star(s)\,ds\right)^q dt.
\]
We perform the change of variables $r=\mu_{u^*}(t)$, so that the above equation turns into
\[
\int_E |\nabla u|^q w^q\,d\mu\le \int_0^{\mu(E)} \left(h(r)^{-1}\int_0^{r}f^\star(s)\,ds\right)^q dr.
\]
By a straightforward inspection of those steps we notice that $v$ satisfies 
\[
\int_{R_E} |\nabla v|^q w^q\,d\mu=\int_0^\infty
(-\mu_{v}'(t))\left(h(\mu_{v}(t))^{-1}\int_0^{\mu_{v}(t)}f^\star(s)\,ds\right)^q dt;
\]
By performing the change of variables $r=\mu_v(t)$ we find
\[
\int_{R_E} |\nabla v|^q w^q\,d\mu= \int_0^{\mu(R_E)} \left(h(r)^{-1}\int_0^{r}f^\star(s)\,ds\right)^q dr.
\]
Since $\mu(E)=\mu(R_E)$ we get the desired result.

We are left to prove that the function $v$ given by \eqref{v} is a solution of problem
\eqref{problemsym}. We start by noticing that
%
%
%
equation \eqref{stima6} suggests how to derive \eqref{v}: indeed, as we
pointed out, any solution $v$ of \eqref{problemsym} such that $v=v^*$ satisfies
\[
 \frac{\int_0^{\mu_{v}(t)}f^\star(s)\,ds}{h^2(\mu_{v}(t))}\mu_{v}'(t)=-1.
\]
By integrating both sides between $0$ and $r$ we obtain
\[
\int_0^r \frac{\int_0^{\mu_{v}(t)}f^\star(s)\,ds}{h^2(\mu_{v}(t))} \mu_{v}'(t)\,dt=-r.
\]
so that, by performing the change of variables $m=\mu_v(t)$, we get
\[
\int_{\mu_v(r)}^{\mu(R_E)}\frac{\int_0^{m}f^\star(s)\,ds}{h^2(m)} dm=r\]
which is equivalent to
\[
v(z,z')=\int_{\mu\{x_1>z\}}^{\mu(R_E)}\frac{\int_0^{m}f^\star(s)\,ds}{h^2(m)} dm,
\]
that is \eqref{v}.
Notice that $v$ is strictly decreasing and belongs to
$C_{\mathrm{loc}}^{1,1}(R_E)$. Indeed, recalling \eqref{acca} one can explicitly compute
\[
\nabla v(z,z')=e_1 \frac{\pa v}{\pa z}(z,z')=-e_1 \frac{\int_0^{\mu\{x_1>z\}}f^\star(s)\,ds}{w^2(z)
\int_{R^{d-1}}e^{V(z,x')}\,dx'},
\]
where $e_1=(1,0,\dots,0)\in\R^d$.
Since $f^\star$ is a decreasing and locally integrable function, then $f^\star\in L_{\mathrm{loc}}^\infty(\R)$;
thus, being $z\mapsto \mu(\{x_1>z\})$ $C^1$-regular, we get that $\int_0^{\mu\{x_1>z\}}f^\star(s)\,ds$ is a
locally Lipschitz 
function.
Moreover the denominator is locally Lipschitz as well, and locally bounded away from zero.
Hence we have that $\nabla v$ is locally Lipschitz.
%
Thus, recalling that $\partial_1 V$ depends only on the first variable $x_1$ it is possible to explicitly compute
the divergence of $w^2\nabla v e^V$ and check that it satisfies \eqref{problemsym}. This concludes the proof of
the theorem.
\end{proof}

\section*{Acknowledgements}
B. Ruffini was partially supported by the PRIN 2010-2011 ``Calculus of Variations''. The authors thank L.
Brasco and G. De Philippis for useful discussions on the topic. They are also grateful to the anonymous referee
for several suggestions and remarks.

\end{document}